\documentclass[a4paper]{scrartcl}
\usepackage{mathtools,amsthm}
\usepackage[mathscr]{euscript}
\usepackage{tikz}
\usepackage{dsfont}
\usepackage{xcolor}

\usetikzlibrary{shapes}

\numberwithin{equation}{section}

\definecolor{light}{gray}{.75}
\newcommand{\zero}{\mathbf{0}}
\newcommand{\R}{\mathds{R}}
\newcommand{\N}{\mathds{N}}
\newcommand{\co}{\mathrm{co}}

\newcommand{\PP}{\mathscr{P}} %Positive points
\newcommand{\NN}{\mathscr{N}} %Negative points
\newcommand{\HH}{\mathscr{H}} %Hyperplane
\newcommand{\x}{\mathbf{x}} %Vector
\newcommand{\y}{\mathbf{y}} %Vector
\newcommand{\M}{\mathbf{M}} %Vector

\newtheorem{theorem}{Theorem}[section]

\newtheorem{lemma}{Lemma}[section]

\theoremstyle{definition}
\newtheorem{remark}{Remark}[section]
\newtheorem{definition}{Definition}[section]
\newtheorem{example}{Example}[section]

\def\clap#1{\hbox to 0pt{\hss#1\hss}}

  \title{Chebyshev approximation for multivariate functions}
  
\author{{\bf Nadezda Sukhorukova}, Swinburne University of Technology, Australia,\\ {nsukhorukova@swin.edu.au},\\
            {\bf Julien Ugon}, Centre for Informatics and Applied Optimization, \\ Federation University Australia,  j.ugon@federation.edu.au,\\
            {\bf David Yost}, Centre for Informatics and Applied Optimization,  \\ Federation University Australia, d.yost@federation.edu.au\\}
  \date{}

\begin{document}
  \maketitle
  \begin{abstract}
   In this paper, we derive  optimality conditions (Chebyshev approximation) for multivariate functions. The theory of Chebyshev (uniform) approximation for univariate functions is very elegant. The optimality conditions are based on the notion of alternance (maximal deviation points with alternating deviation signs). It is not very straightforward, however, how to extend the notion of alternance to the case of multivariate functions.   There have been several attempts to extend the theory of Chebyshev approximation to the case of multivariate functions. We propose an alternative approach, which is based on the notion of
convexity and nonsmooth analysis.
% Our optimality conditions are equivalent to the inf-stationarity in the sense of Demyanov-Rubinov. In general, this is a necessary optimality condition, however, under some additional assumptions (convexity) this condition is a necessary and sufficient optimality condition.
     \end{abstract}
\textbf {Math subject classification}[2010]: {49J52, 90C26, 41A15, 41A50}
\section{Introduction}\label{sec:introduction}

In this paper, we obtain Chebyshev (uniform) approximation optimality conditions for multivariate functions. The theory of Chebyshev approximation for univariate functions was built in the late nineteenth (Chebyshev) and twentieth century (just to name a few \cite{nurnberger, rice67, Schumaker68}). In most cases, the authors were working on polynomial and polynomial spline approximations, however, other types of  functions (e.g., trigonometric polynomials) have also been used. This  theory is very elegant. In most cases, the optimality conditions are based on the notion of alternance (that is, maximal deviation points with alternating deviation signs).

There have been several attempts to extend this theory to the case of multivariate functions. One of them is \cite{rice63}. In this paper the author underlines the fact that the main difficulty is to extend the notion of alternance to the case of more that one variable.
The main obstacle is that it is not very easy to extend the notion of monotonicity to the case of several variables. Also, several studies have been done in the area of multivarite interpolation~\cite{Nurn_Davydov98multivar_interpolation, Nurnberger_multivariate_inter}, where triangulation based approaches were used to extend the notion of polynomial splines to the case of multivariate function. 

The objective function, appearing in Chebyshev approximation optimisation problems is nonsmooth (minimisation of the maximal absolute deviation). Therefore, it is logical to use nonsmooth optimisation techniques to tackle this problem. In this paper we propose an approach, which is based on the notion of convex function subdifferentials~\cite{Rockafellar70}. Subdifferentials can be considered as generalisation of the notion of gradients for convex nondifferential functions. 

%Our optimality conditions are equivalent to the inf-stationarity in the sense of Demyanov-Rubinov. In general, this is a necessary optimality condition, however, under some additional assumptions (i.e., convexity of the objective function) this condition is a necessary and sufficient optimality condition.

Another powerful nonsmooth analysis tool, quasidifferentiability, has been successfully applied to improve the existing optimality conditions in the case of free knot polynomial spline approximation~\cite{su13}. In particular, these techniques allowed the researchers to overcome the difficulties highlighted in~\cite{FreeKnotsOpenProblem96} as ``the existing optimisation tools are not adapted to this problem, due to its nonconvex and nonsmooth nature''. Quasidifferentiability is one of the modern nonsmooth optimisation approaches, which is not well-known outside of the optimisation research community. This technique, however, is very powerful and enables one to work efficiently with nonsmooth and nonconvex functions. In particular, it has been successfully applied to improve the existing~\cite{nurnberger} necessary optimality conditions for the case of free knots polynomial splines~\cite{su10,su13}.

Quasidifferentiability can be treated as a generalisation of subdifferentials to the case of nonconvex functions. In this study, however,   the objective function is convex and therefore, it is sufficient to use subdifferentials, since this tool is simpler and easier to use and understand.

 The paper is organised as follows. In section~\ref{sec:optimality conditions} we present the most relevant results from the theory of convex and nonsmooth analysis, that are essential to obtain our optimality conditions.
 Then, in the same  section, we investigate the extremum properties of the objective function, appearing in Chebyshev approximation problems, from the points of view of convexity and nonsmooth analysis. In section~\ref{sec:optimality conditions} we obtain our main results. In section~\ref{seq:relation_with_existing_univariate} we demostrate the relation between the obtained results and the classical ones (alternance-based). Then in section~\ref{seq:relation_with_existing_multivariate} we demonstrate the relation with other optimality results (multivariate case), obtained by J. Rice~\cite{rice63}. Finally, in section~\ref{sec:conclusions} we draw our conclusions and underline further research directions.

 \section{Optimality conditions}\label{sec:optimality conditions}

 \subsection{Convexity of the objective}\label{ssec:convexObjective}
 Let us now formulate the objective function. Suppose that a continuous function $f(\x)$ is to be approximated by a function
 \begin{equation}\label{eq:model_function}
 L(A,\x)=a_0+\sum_{i=1}^{n}a_ig_i(\x),
 \end{equation}
 where $g_i(\x)$ are the basis functions and the multipliers $A = (a_1,\dots,a_n)$ are the corresponding coefficients. At a point \(x\) the deviation between the function \(f\) and the approximation is:
 \begin{equation}
   d(A,x) = |f(\x) - L(A,\x)|.
\end{equation}
   \label{eq:deviation}
 Then we can define the uniform approximation error over the set \(Q\) by
 \begin{equation}
   \label{eq:uniformdeviation}
\Psi(A)=\sup_{\x\in Q} \max\{f(\x)-a_0-\sum_{i=1}^{n}a_ig_i(\x),a_0+\sum_{i=1}^{n}a_ig_i(\x)-f(\x)\}.
\end{equation}
 The approximation problem can be formulated as follows.
 \begin{equation}\label{eq:obj_fun_con}
   \mathrm{minimise~}\Psi(A) \mathrm{~subject~to~} A\in \R^{n+1}.
 \end{equation}

 We will consider two cases.
 \begin{description}
   \item[Continuous case] the set $Q$ is a hyperbox, such that $c_i\leq x_i\leq d_i,~i=1,\dots,d$.
   \item[Discrete case] the set \(Q\) is a finite set of points.
 \end{description}

 Since the function \(L(A,\x)\) is linear in \(A\), the approximation error function \(\Psi(A)\), as the supremum of affine functions, is convex. Furthermore, its subdifferential at a point \(A\) is trivially obtained using the active affine functions in the supremum:
 \begin{equation}
   \label{eq:subdifferentialObjective}
   \partial \Psi(A) = \co\left\{ \begin{pmatrix}
1\\
g_1(\x)\\
g_2(\x)\\
\vdots \\
g_n(\x)
\end{pmatrix}: \x \in E^+,-\begin{pmatrix}
1\\
g_1(\x)\\
g_2(\x)\\
\vdots \\
g_n(\x)
\end{pmatrix}: \x\in E^-\right\},
\end{equation}
 where \(E^+\) and \(E^-\) are respectively the points of maximal positive and negative deviation:
 \begin{align*}
   E^+ &= \Big\{\x\in Q:  L(A,\x) - f(\x) = \max_{\y\in Q} d(A,\y)\Big\}\\
   E^- &= \Big\{\x\in Q: f(\x) - L(A,\x) = \max_{\y\in Q} d(A,\y)\Big\}.
 \end{align*}

\subsection{Optimality conditions: general case}\label{ssec:opt_general}

\begin{theorem}\label{thm:main}
  The convex hulls of the vectors $(g_1(\x),\dots,g_n(\x))^T,$ built over positive and negative maximal deviation points intersect.
\end{theorem}
\begin{proof}
The necessary and sufficient condition for the optimality of a vector \(A^*\) for the convex problem \eqref{eq:obj_fun_con} is
\[
  \zero_{n+1} \in \partial \Psi(A^*).
\]
Note that due to Caratheodory´s theorem, $\zero$ can be constructed as a convex combination of a finite number of points. Namely, since the dimension of the corresponding space is $n+1$, it can be done using at most $n+2$ points.
  
By the formulation of the subdifferential of \(\Psi\) \(\partial\Psi(A)\) \eqref{eq:subdifferentialObjective}, there exists a nonnegative  number \(\gamma \leq 1\) and two vectors 
\[
  g^+ \in \co\left\{ \begin{pmatrix}
1\\
g_1(\x)\\
g_2(\x)\\
\vdots \\
g_n(\x)
\end{pmatrix}: \x \in E^+\right\}, \mathrm{~and~}
  g^- \in \co\left\{ \begin{pmatrix}
1\\
g_1(\x)\\
g_2(\x)\\
\vdots \\
g_n(\x)
\end{pmatrix}: \x \in E^-\right\}
\]
such that \(\zero = \gamma g^+ - (1-\gamma) g^-\). Noticing that the first coordinates \(g^+_1 = g^-_1 = 1\), we see that \(\gamma = \frac{1}{2}\). This means that \(g^+ - g^- = 0\). This happens if and only if
 \begin{equation}\label{eq:opt_main2}
 \co\left\{
\left(
\begin{matrix}
1\\
g_1(\x)\\
g_2(\x)\\
\vdots
\\
g_n(\x)\\
\end{matrix}
\right): \x \in E^+
 \right
  \}\cap
  \co\left\{
\left(
\begin{matrix}
1\\
g_1(\x)\\
g_2(\x)\\
\vdots
\\
g_n(\x)\\
\end{matrix}
\right): \x \in E^-
 \right \}\ne\emptyset.
 \end{equation}
 As noted before, the first coordinates of all these vectors are the same, and therefore the theorem is true.
 \end{proof}

\subsection{Optimality conditions for multivariate linear functions}
\label{ssec:opt_linear_multi}
In the case of linear functions (multidimensional case), $n=d$ and  the functions $g_i=x_i,$ $i=1,\dots,n$.  Then theorem~\ref{thm:main} can be formulated as follows.
\begin{theorem}\label{thm:main_lin}
The convex hull of the maximal deviation points with positive deviation and convex hull of the maximal deviation points with negative deviation have common points.
\end{theorem}

Theorem~\ref{thm:main_lin} can be considered as an alternative formulation to the necessary and sufficient optimality conditions that are based on the notion of alternance. Clearly, theorem~\ref{thm:main_lin} can be used in univariate cases, since the location of the alternance points insures the common points for the corresponding convex hulls, constructed over the maximal deviation points with positive and negative deviations respectively.

Note that in general $d\leq n$.

\subsection{Optimality conditions for multivariate polynomial (non-linear)  functions}
\label{ssec:opt_polynomial_multi}

We start from introducing the following definitions and notation.

\begin{definition}
An exponent vector
$$e=(e_1,\dots,e_d)\in \R^d,~e_i\in \N,~i=1,\dots,d$$ for $\x\in\R^d$
defines a {\em monomial}
$$\x^e=x_1^{e_1} x_2^{e_2} \dots x_d^{e_d}.$$
\end{definition}

\begin{definition}
A product $c\x^e,$ where $c\ne 0$ is called the term, then a multivariate polynomial is a sum of a finite number of terms.
\end{definition}

\begin{definition}
The degree of a monomial $\x^e$ is the sum of the components of $e$:
$$\deg(\x^e)=\sum_{i=1}^{d}e_i.$$
\end{definition}

\begin{definition}
The degree of a polynomial is the largest degree of the composing it monomials.
\end{definition}

Let us consider some essential properties of polynomials and monomials.
\begin{enumerate}
\item For any exponential $e=(e_1,\dots e_d),$ such that $\deg \x^e=m$ the degree of the monomial $\x^{\tilde{e}}=m+1,$ where the exponential $\tilde{e}$ has been obtained from $e$ by substituting one of its components $e_k$ by $e_k+1,$ $k=1,\dots,n.$ Any monomial of degree $m+1$ can be obtained in such a way. In general, there may be more than one way to do this.
\item For any exponential $e=(e_1,\dots e_d),$ such that $\deg \x^e=m$ the degree of the monomial $\x^{\tilde{e}}=m-1,$ where the exponential $\tilde{e}$ has been obtained from $e$ by substituting one of its components positive components $e_k>0$ by $e_k-1,$ $k=1,\dots,n.$ Any monomial of degree $m-1$ can be obtained in such a way. In general, there may be more than one way to do this.
\end{enumerate}

 In general, a polynomial of degree $m$ can be obtained as follows:
 \begin{equation}\label{eq:polynomials}
 P^m(x)=a_0+\sum_{i=1}^{n}a_iM_i(\x,e),
\end{equation}
where $a_i$ are the coefficients and $g_i=M_i$ are the monomials, such that $\deg{M_i}\leq m$ and there exists a monomial $M_k,$ such that $\deg(M_k)=m.$ Any polynomials $P^m$ from~(\ref{eq:polynomials}) can be presented as the sum of a lower degree polynomials ($m-1$ or less) and a finite number of terms that correspond to the monomials of degree $m$. The following lemma holds.

\begin{lemma}\label{lem:monomial}
Consider two sets of non-negative coefficients
\begin{itemize}
\item $\alpha_i\geq ,~i=1,\dots,n$ such that $\sum_{i=1}^{n}\alpha_i=1$;
\item $\beta_i\geq ,~i=1,\dots,n$ such that $\sum_{i=1}^{n}\beta_i=1$.
\end{itemize}
If
\begin{equation}\label{eq:lem_1}
\sum_{i=1}^{n}\alpha_ia_ix_i=\sum_{i=1}^{n}\beta_ib_iy_i
\end{equation}
and \begin{equation}\label{eq:lem_2}
\sum_{i=1}^{n}\alpha_ia_i=\sum_{i=1}^{n}\beta_ib_i
\end{equation}
then for any scalar $\delta$ the following equality holds
\begin{equation}\label{eq:lem_3}
\sum_{i=1}^{n}\alpha_ia_i(x_i-\delta)=\sum_{i=1}^{n}\beta_ib_i(y_i-\delta).
\end{equation}
\end{lemma}
\begin{proof}
\begin{align*}
\sum_{i=1}^{n}\alpha_ia_i(x_i-\delta)=&\sum_{i=1}^{n}\alpha_ia_ix_i-\delta\sum_{i=1}^{n}\alpha_ia_i\\
  =&\sum_{i=1}^{n}\beta_ib_i-\delta\sum_{i=1}^{n}\beta_ib_iy_i\\
  =&\sum_{i=1}^{n}\beta_ib_i(y_i-\delta).
\end{align*}
\end{proof}

Theorem~\ref{thm:main}  can be used to formulate our necessary and sufficient optimality conditions for multivariate polynomial approximations. Each $g_i$ corresponds to a monomial and, in general, we need to use all the possibilities to construct the monomials, keeping the corresponding monomial degree at most $m.$

Note that due to lemma~\ref{lem:monomial} one can assume that all the $x_i$ in the monomials are non-negative, since $\delta$ can be chosen as
$$\min\{\min_{i=1,\dots,d}x_i,\min_{i=1,\dots,d}y_i\}.$$  Then a necessary and sufficient optimality condition can be formulated as follows.

\begin{theorem}\label{thm:pol}
A polynomial of degree $m$ is optimal if and only if there exists a pair of sets of non-negative coefficients (with at least one positive coefficient in each set)
$$\alpha_{i+},~\alpha_{i-},~{i+},{i-}=1,\dots,n+2,~\sum_{i+=1}^{n+2}\alpha_{i+}=\sum_{i-=1}^{n+2}\alpha_{i-}=1,$$
such that for any monomial~$M_j,~j=1,\dots,n$ of degree at most $m$ the following equality holds
\begin{equation}\label{eq:monom}
\sum_{i+=1}^{n+2}\alpha_{i+}M_j(\x^{i+})=\sum_{i-=1}^{n+2}\alpha_{i-}M_j(\x^{i-}),
\end{equation}
where $\x^{i+}$ and $\x^{i-}$ are maximal deviation points with positive and negative deviation sign respectively.
\end{theorem}
\begin{proof}
The proof is obvious, since this is a reformulation of theorem~\ref{thm:main} for the case of polynomials. Similar to theoremref{thm:main}, Caratheodory´s theorem is used to prove that it is possible to do all the necessary constructions using at most $n+2$ points.
\end{proof}

Note that any monomial $\prod_{i=1}^{l} x_i^{e_i}$ of degree $m\geq 1,$ such that $m=\sum_{i=1}^{l}e_i$ can be presented as a product of  a lower degree monomial and $x_i$ (assume that $e_i\geq 1$). Therefore, theorem~\ref{thm:pol} can be also formulated as follows.
\begin{theorem}\label{thm:pol1}
A polynomial of degree $m$ is optimal if and only if there exists a pair of sets of non-negative coefficients (with at least one positive coefficient in each set)
$$\alpha_{i+},~\alpha_{i-},~{i+},{i-}=1,\dots,n+2,~\sum_{i+=1}^{n+2}\alpha_{i+}=\sum_{i-=1}^{n+2}\alpha_{i-}=1,$$
such that for any monomial $M$ of degree at most $m-1$  the following equality holds
\begin{equation}\label{eq:monom1}
\sum_{i+=1}^{n+2}\alpha_{i+}M(\x^{i+})\x^{i+}=\sum_{i-=1}^{n+2}\alpha_{i-}M(\x^{i-})\x^{i-},
\end{equation}
where $\x^{i+}$ and $\x^{i-}$ are maximal deviation points with positive and negative deviation sign respectively.
\end{theorem}
Theorem~\ref{thm:pol1} tells us that the convex hulls built over positive and negative maximal deviation points intersect. Moreover, the linear combinations of these two sets with the coefficients $\alpha_{i+}M(\x_{i+})$ and $\alpha_{i-}M(\x_{i-})$, where $M(\x)$ is a monomial of degree at most $m-1$.

Note that these linear combinations are not necessary convex, since some of the coefficients $M(\x_{i+})$ and $M(\x_{i-})$ may be negative. One can make these coefficients non-negative by applying lemma~\ref{lem:monomial}. This can be achieved in a number of ways. For example, for a monomial 
$$M(\x)=x_1^{e_1}\times x_2^{e_2}\times\dots\times x_d^{e_d},$$
where $e_1,\dots, e_d$ are non-negative integers, representing the degree, such that $$\sum_{j=1}^{d}e_j\leq m-1,$$ apply lemma~\ref{lem:monomial}, where $\delta$ is chosen as follows  
\begin{equation}\label{eq:delta_min}
\delta=\min_{j=1,\dots, d, e_j>0}x_j
\end{equation}
or 
\begin{equation}\label{eq:delta_max}
\delta=-\max_{j=1,\dots, d, e_j>0}x_j.
\end{equation}

Necessary and sufficient optimality conditions formulated in theorems~\ref{thm:main}-\ref{thm:pol1} are not very easy to verify. In this paper we develop a necessary optimality condition that is more practical. 

Assume that our polynomial approximation is optimal and consider all the monomials of degree $m-1$
\begin{equation}
M^{m-1}=x_1^{e_1}x_2^{e_2}\dots x_l^{e_l},~\text{such~that}~\sum_{i=1}^{d}e_i=m-1.\end{equation}
Then all the monomials of degree $m$ can be obtained by multiplying one of the monomials of degree $m-1$ by one of the coordinated of $\x=(x_1,\dots,x_d)^T.$ Therefore, there exists a pair of sets of positive coefficients (remove zero coefficients for simplicity) 
\begin{equation}\label{eq:coefficients}
\alpha_{i+},~i+=1,\dots,N^+,~\alpha_{i-}=1,\dots,N^-,
\end{equation}  
 such that
 \begin{equation}
 \sum_{i+=1}^{N+}\alpha_{i+}M^{m-1}(1,\x_{i+})^T=\sum_{i-=1}^{N-}\alpha_{i-}M^{m-1}(1,\x_{i-})^T
 \end{equation}
Assume that in the monomial $M^{m-1}$ there exists $j,$ such that $e_j>0.$ Consider 
$$\delta^j_{min}=\min\{\min_{i+=1,\dots,N+}x^{i+}_j,\min_{i-=1,\dots,N-}x^{i-}_j\}$$
and
$$\delta^j_{max}=\max\{\max_{i+=1,\dots,N+}x^{i+}_j,\max_{i-=1,\dots,N-}x^{i-}_j\},$$
where $x^{i+}_j$ is the $j-$th coordinate of the maximal deviation points $\x_{i+}$ and $x^{i-}_j$ is the $j-$th coordinate of the maximal deviation points $\x_{i-}.$ 
Apply lemma~\ref{lem:monomial} with $\delta=\delta^j_{min}$ or $\delta=\delta^j_{max}$ and remove all the maximal deviation points with the minimal (maximal) value for the $j-$th coordinate (there may be more than one point), since the corresponding monomial is zero.  At the end of this process, the convex hulls of the remaining maximal deviation points should intersect or all the maximal deviation points are removed.

%Assume now that the last $n_1$ components of the vector $(1,g_1,\dots,g_n)^T$ (the last $n_1$ monomials) correspond to the monomials of degree $m,$ then the components that correspond to the first $n-n_1$ monomials correspond to the polynomials of degree $m-1$. If
%%$$\alpha_{i+},~\text{and}~\beta_{i-},~{i+},{i-}=1,\dots,n+2$$
%are a set of convex coefficients from~(\ref{eq:monom}), then, since $x_i\geq 0$, $i=1,\dots,l$ (lemma~\ref{lem:monomial}), one can conclude that for monomials of degree~$m-1$ and lower expression~(\ref{eq:monom}) also holds for the coefficients
%$$\tilde{\alpha}_{i+}={\alpha_{i+}x^{i+}\over \sum_{i+=1}^{n+2}\alpha_{i+}x^{i+}}$$
%and
%$$\tilde{\beta}_{i-}={\beta_{i-}x^{i-}\over \sum_{i-=1}^{n+2}\beta_{i-}x^{i-}}$$
%if
%\begin{equation}\label{eq:degen_case}
%\sum_{i+=1}^{n+2}\alpha_{i+}x^{i+}=\sum_{i-=1}^{n+2}\beta_{i-}x^{i-}\ne 0.
%\end{equation}
%
%If~(\ref{eq:degen_case}) does not hold, then we are in a special situation, where all the maximal deviation points that correspond to positive $\alpha_{i+}$ and $\beta_{i-}$ are zero. Therefore they do not contribute in the distinguishing between optimal and non-optimal polynomials approximations and can be excluded from further consideration.

The following algorithm can be used to verify this necessary optimality condition. 

\begin{center}{\Large\bf
Necessary optimality conditions verification for best polynomial Chebyshev approximation: degree $m>1,$ multivariate case.}
\end{center}
\begin{enumerate}
\item[Step 1] Identify maximal deviation points that correspond to positive and negative deviation:
$$E^+=\{x^{i+},~i=1,\dots,N^+\};~E^-=\{x^{i-},~i=1,\dots,N^-\}.$$
\item[Step 2] For each dimension $k:$ $k=1,\dots,l$ identify
$$\delta=\min \{ \min_{j=1,\dots,l}x^{i+}_j,\min_{j=1,\dots,l}x^{i-}_j\}$$
or
$$\delta=\max \{ \max_{j=1,\dots,l}x^{i+}_j,\max_{j=1,\dots,l}x^{i-}_j\}$$
where $x^{i+}_j$ and $x^{i-}_j$ are the $j-$th coordinates of $\x^{i+}$ and $\x^{i-}$ respectively.
\item[Step 3] Apply the following coordinate transformation (to transform the coordinates of the maximal deviation points to non-negative numbers):
$$\tilde{x}^{i+}_j=x^{i+}_j-\delta;~\tilde{x}^{i-}_j=x^{i-}_j-\delta.$$
\item[Step 4] Points reduction
Remove  maximal deviation points  whose  updated coordinates have a zero at the corresponding  coordinate and assign $m_{new}=m-1$.

If $m_{new}>1$ and the remaining sets of maximal deviation (positive and negative) are non-empty GO TO Step 1 for the corresponding lower degree polynomials approximation optimality verification, $m=m_{new}$).

Otherwise GO TO the final step of the algorithm.
\item[Step 5] If the remaining maximal deviation sets are non-empty and the conditions of theorem~\ref{thm:main_lin}  are not  satisfied then the original polynomials is not optimal.
\end{enumerate}

There are two main advantages of this procedure.
\begin{enumerate}
\item It demonstrates how the concept of alternance can be generalised  to the case of multivariate functions.
\item It is based on the verification whether two convex sets are intersecting or not, but since $l\leq n$ it is much easier to verify it after applying the algorithm.
\end{enumerate}

However, the this condition is a necessary optimality condition, but not sufficient. In section~{sec:counterexample} we will present an interesting counterexample, that demonstrates that even more general alternating conditions still remain only necessary optimality conditions for multidimentional functions.

Note that theorem~\ref{thm:main} can be also applied to verify optimality (necessary and sufficient optimality condition). In this case one needs to check if two convex sets are intersecting  in $\R^{n}.$ The above algorithm requires
to check if two convex sets are intersecting  in $\R^{d}$ (considerably lower dimension), however, it only verifies a necessary optimality condition. 

\section{Counterexample}\label{sec:counterexample}
Consider new notation.
\begin{enumerate}
\item [] $\M^m$ is a vector that contains all the monomials of degree $m$ or less;
\item [] $\alpha_{\x}$ is a nonnegative coefficient associated with $\x$.
\end{enumerate}

\begin{color}{blue}
  It is clear that the system:
  \[
    \sum_{\x\in E^+} \alpha_\x \M^m(\x) = \sum_{\x\in E^-} \alpha_\x \M^m(\x) 
  \]
  is equivalent to:
  \begin{align}\refstepcounter{equation}
    \sum_{\x\in E^+} \alpha_\x M^{m-1}(x) &= \sum_{\x\in E^-} \alpha_\x \M^{m-1}(x) \tag{\theequation(0)}\label{eq:n-10}\\
    \sum_{\x\in E^+} \alpha_\x x_i \M^{m-1}(x) &= \sum_{\x\in E^-} x_i\alpha_\x M^{m-1}(x)  & i=1,\ldots,d. \tag{\theequation(i)}\label{eq:n-1i}
  \end{align}

  Consider any vector \({\bf u}\in \R^d\) and scalar \(a\in \R\), and let
  \begin{align*}
    \HH^+ &= \{\x\in \R^d : \langle {\bf u},\x\rangle - a> 0\}\\
    \HH^- &= \{\x\in \R^d : \langle {\bf u},\x\rangle - a< 0\}
  \end{align*}

  and consider the equation \(\sum_{i=1}^d u_i\times \eqref{eq:n-1i} - a\times \eqref{eq:n-10}\). We find that
  \[
    \sum_{\x\in E^+} \alpha_\x (\langle {\bf u},\x\rangle -a )\M^{m-1}(\x) = \sum_{\x\in E^-} \alpha_\x (\langle {\bf u},\x\rangle -a )\M^{m-1}(x)
  \]
  Define 
  \begin{align*}
    A^+ &= \sum_{\y\in (E^+\cap \HH^+)\cup(E^-\cap \HH^-)}\alpha_\y(\langle {\bf u},\y\rangle -a )\\
    A^- &= \sum_{\y\in (E^-\cap \HH^+)\cup(E^+\cap \HH^-)}\alpha_\y(\langle {\bf u},\y\rangle -a )
  \end{align*}
  and
  \begin{align*}
    \hat{\alpha}_\x &= \frac{\alpha_\x }{A^+} & \text{if } \x\in (E^+\cap \HH^+)\cup(E^-\cap \HH^-)\\
    \hat{\alpha}_\x &= \frac{\alpha_\x }{A^-} & \text{if } \x\in (E^-\cap \HH^+)\cup(E^+\cap \HH^-)
  \end{align*}
  Then, we have:
  \begin{equation}
    \sum_{\x\in (E^+\cap \HH^+)\cup(E^-\cap \HH^-)} \hat{\alpha}_\x \M^{m-1}(x) =\sum_{\x\in (E^-\cap \HH^+)\cup(E^+\cap \HH^-)} \hat{\alpha}_\x \M^{m-1}(\x) \label{eq:onelessdegree}
  \end{equation}
  Note that Formula~\eqref{eq:onelessdegree} is similar to Formula~\eqref{eq:monom} with degree \(m-1\). Thus, the above result means that if one runs any hyperplane and inverts the signs on one side of this hyperplane, the formula holds for degree \(m-1\). This implies the following corollaries:

  \begin{enumerate}
    \item Since any \(k\leq d\) points define (not necessarily uniquely) a hyperplane, by setting the pair \(({\bf u},a)\) to be defining this hyperplane, the result for degree \(m-1\) applies, after setting the signs accordingly.
    \item In particular, for if one choses these \(k\) points as the vertices defining a \(k-1\)-face of the polytope \(P=\co\{\x\in E\}\), then all remaining points lie on the same side of the hyperplane. Therefore, by removing any \(k\) face (and facets in particular) of the polytope \(P\), the result holds for degree \(m-1\).
    \item\label{item:removefacets} By iteratively removing any \(m-1\) facets as described above, the remaining polytopes \(P^+\) and \(P^-\) must intersect (that is, the result holds for degree 1).
    \item\label{item:removepoints} Similarly, it is possible to remove any \((m-1)d\) points and, updating the signs accordingly, the remaining polytopes \(P^+\) and \(P^-\) must intersect.
  \end{enumerate}

  \begin{remark}
      It is easy to verify that Condition~\ref{item:removefacets} is exactly equivalent to the alternation criterion in the univariate case. Indeed, after removing \(m-1\) points (which are also facets in the univariate case), there needs to remain at least 3 alternating points to ensure intersection of the remaining polytopes. This means that there must be at least \(m+2=m-1+3\) points, which can be shown to alternate by removing the adequate \(m-1\) points.

      Similarly, Condition~\ref{item:removefacets} is trivial for degree \(m=1\) polynomial approximation, since all these conditions are equivalent to the intersection between \(E^+\) and \(E^-\) being nonempty.
  \end{remark}

  Generally, however, the condition~\eqref{eq:onelessdegree} is not sufficient, as is illustrated in the following example:
  \begin{example}
      Consider the quadratic approximation case and assume \(E^+ = \co\{(0,0),
	  (1,2), (2,0)\}\) and \(E^- = \{(0,1),(1,-1),(2,1)\}\). These extreme deviation points do not satisfy the condition~\eqref{eq:monom}. Yet, they satisfy the condition~\eqref{eq:onelessdegree}. The sets are illustrated in the figure below:

      
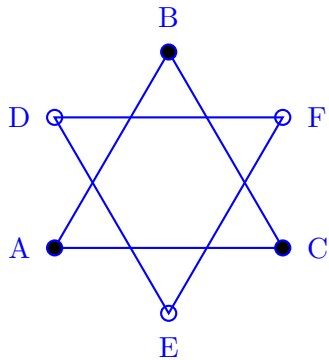
\begin{center}
	  \begin{tikzpicture}[scale=3,every node/.style={draw,circle,outer sep=1pt, inner sep=2pt}]
	      \coordinate (A) at (0,0) ;
	      \coordinate (B) at (0.5,0.866) ;
	      \coordinate (C) at (1,0) ;
	      \coordinate (D) at (0,0.577) ;
	      \coordinate (E) at (0.5,-0.289);
	      \coordinate (F) at (1,0.577) ;
	      \draw[thick,every node/.append style={fill=black}] (A) node[label=left:A] {} -- (B) node[label=above:B] {} -- (C) node[label=right:C] {} -- cycle;
	      \draw[thick] (D) node[label=left:D] {} -- (E) node[label=below:E] {}  -- (F) node[label=right:F] {} -- cycle;
	  \end{tikzpicture}
      \end{center}

      and below we show that for any hyperplane cutting through the polytope \(\co(E^+\cup E^-)\) the intersection exists (obvious symmetries are omitted).

    
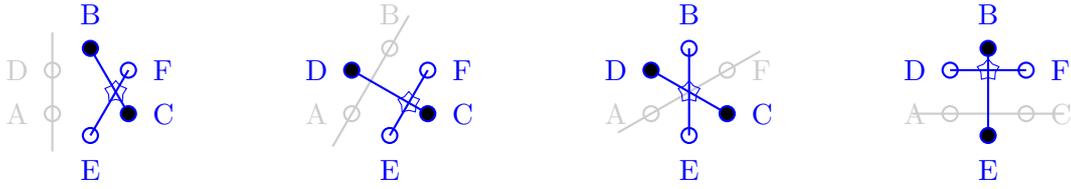
\begin{tikzpicture}[every node/.style={draw,circle,outer sep=1pt, inner sep=2pt},
    extended line/.style={shorten >=-#1,shorten <=-#1},
	,extended line/.default=0.5cm]
	\coordinate (A) at (0,0) ;
	\coordinate (B) at (0.5,0.866) ;
	\coordinate (C) at (1,0) ;
	\coordinate (D) at (0,0.577) ;
	\coordinate (E) at (0.5,-0.289);
	\coordinate (F) at (1,0.577) ;
	\draw[extended line,thick,black!20] (A) node[label=left:A] {} -- (D) node[label=left:D] {};
	\draw[thick,every node/.append style={fill=black}] (B) node[label=above:B] {} -- (C) node[label=right:C] {} -- cycle;
	\draw[thick] (E) node[label=below:E] {}  -- (F) node[label=right:F] {} -- cycle;
	\node[star] at (intersection of B--C and E--F) {};
    \end{tikzpicture}
    \hfill
    \begin{tikzpicture}[every node/.style={draw,circle,outer sep=1pt, inner sep=2pt},
    extended line/.style={shorten >=-#1,shorten <=-#1},
	,extended line/.default=0.5cm]
	\draw[extended line,thick,black!20] (A) node[label=left:A] {} -- (B) node[label=above:B] {};
	\draw[thick,every node/.append style={fill=black}] (D) node[label=left:D] {} -- (C) node[label=right:C] {} -- cycle;
	\draw[thick] (E) node[label=below:E] {}  -- (F) node[label=right:F] {} -- cycle;
	\node[star] at (intersection of D--C and E--F) {};
    \end{tikzpicture}
    \hfill
    \begin{tikzpicture}[every node/.style={draw,circle,outer sep=1pt, inner sep=2pt},
    extended line/.style={shorten >=-#1,shorten <=-#1},
	,extended line/.default=0.5cm]
	\draw[extended line,thick,black!20] (A) node[label=left:A] {} -- (F) node[label=right:F] {};
	\draw[thick,every node/.append style={fill=black}] (D) node[label=left:D] {} -- (C) node[label=right:C] {} -- cycle;
	\draw[thick] (E) node[label=below:E] {}  -- (B) node[label=above:B] {} -- cycle;
	\node[star] at (intersection of D--C and E--B) {};
    \end{tikzpicture}
    \hfill
    \begin{tikzpicture}[every node/.style={draw,circle,outer sep=1pt, inner sep=2pt},
    extended line/.style={shorten >=-#1,shorten <=-#1},
	,extended line/.default=0.5cm]
	\draw[extended line,thick,black!20] (A) node[label=left:A] {} -- (C) node[label=right:C] {};
	\draw[thick,every node/.append style={fill=black}] (B) node[label=above:B] {} -- (E) node[label=below:E] {} -- cycle;
	\draw[thick] (D) node[label=left:D] {}  -- (F) node[label=right:F] {} -- cycle;
	\node[star] at (intersection of D--F and E--B) {};
    \end{tikzpicture}

  \end{example}

\end{color}

\section{Relation with classical (univariate) polynomial approximation results}\label{seq:relation_with_existing_univariate}
\subsection{Linear approximation}
For univariate linear approximation the alternance-based optimality conditions are as follows.
\begin{theorem}
There exist three maximal deviation points, such that the sign of maximal deviation is alternating.
\end{theorem}
Clearly, this theorem is a special case of theorem~\ref{thm:main_lin} for $n=1,$ since one of these three maximal deviation points is located between two others and the deviation sign of this middle point is opposite to the sign at the far left and right points. Therefore, the convex hulls of positive and negative maximal deviation points intersect.
\subsection{Nonlinear approximation}
In the case of higher degree polynomials, the proposed algorithm, at each iteration reduces the degree of the polynomial. If we assume that our dimension is one (univariate approximation), then at each iteration we remove one point (minimal or maximal value) and the remaining $m+2-(m-1)=3$ points satisfy the univariate (alternance) condition.
\subsection{Uneven distribution of maximal deviation points}
In the case of univariate approximation the distribution of positive and negative maximal deviation points, involved in optimality verification, is (almost) even. Indeed, if $n$ is even, then each set contains ${n\over 2}+1$ points, otherwise, one set contains  ${n+1\over 2}$ points and another one ${n+1\over 2}+1$ points. The following example demostrates, that this is not always the case for multivariate approximations.
\begin{example}
Consider a discrete function defined at four isolated points $(1,1)$, $(-1,1)$, $(0,-1)$ and $(0,0)$:
$$
f(1,1)=f(-1,1)=f(0,-1)=0,~\text{and}~f(0,0)=2.
$$
Find a best linear approximation to this function.

It is clear that the plane $L(x,y)=1$ is optimal, since the maximal deviation is attained at $(1,1),~(-1,1),~(0,-1)$ and $(0,0)$. Namely, the maximal deviation magnitude is  equal to~1 and  the deviation signs are negative at the points $(1,1),~(-1,1),~(0,-1)$ and positive at $(0,0)$. Therefore, $\NN$ contains three points, while $\PP$ only one.
 \end{example}
\section{Relation with existing multivariate results}\label{seq:relation_with_existing_multivariate}
In~\cite{rice63} Rice gives necessary and sufficient optimality conditions for multivariate approximation. These results are obtained for a very general class of functions, not necessary polynomials. These conditions are fundamentally important, however, it is not very easy to verify them (even in the case of polynomials). Also their relation with the notion of alternance is not very clear. Before formulating Rice's optimality conditions, we need to introduce the following notation and definitions (\cite{rice63}).

Recall that the set of extremal (maximal deviation points) $E$ is divided into two parts as follows:
$$E^+=\{x|x\in E, f(x)-L(A^*,\x)=\|f(x)-L(A^*,\x)\|_{\infty}\}$$
$$E^-=\{x|x\in E, f(x)-L(A^*,\x)=-\|f(x)-L(A^*,\x)\|_{\infty}\},$$
where $A^*$ is a vector of the parameters and $L(A^*,x)$ is the corresponding approximation, defined as in~(\ref{eq:model_function}).
The elements of $E^+$ and $E^-$ are positive and negative extremal points.

\begin{definition}
The point sets $\PP$ and $\NN$ are said to be isolable if there is an $A,$ such that
$$L(A,\x)>0~\x\in\PP,\quad L(A,\x)<0~\x\in \NN.$$
\end{definition}
\begin{definition}
$\Gamma(A)$ is is called an isolating curve if
$$\Gamma(A)=\{\x| L(A,\x)=0\}.$$
\end{definition}
Therefore, the sets $\PP$ and $\NN$ are isolable, if they lie on opposite sides of an isolating curve $\Gamma(A)$.
\begin{definition}
A subset of extremal points is called a critical point set if its positive and negative parts $\PP$ and $\NN$ are not isolable, but if any point is deleted then $\PP$ and $\NN$ are isolable.
\end{definition}
Rice formulated his necessary and sufficient optimality conditions as follows.
\begin{theorem}(Rice~\cite{rice63}) $L(A^*,\x)$ is a best approximation to $f(\x)$ if and only if the set of extremal points of $L(A^*,\x)-f(\x)$ contains a critical point set.
\end{theorem}
Note that $L(A,\x)$ is linear with respect to $A$ (due to~(\ref{eq:model_function})). Then $\Gamma (A)$ can be interpreted as a linear function (hyperplane). If two convex sets (convex hulls of positive and negative points) are not intersecting, then there is a separating hyperplane, such that these two convex sets lie on opposite sides of this hyperplane.

Note that in our necessary and sufficient optimality conditions we only consider finite subsets of $\PP$ and $\NN$, namely, we only consider the set of at most $n+2$ points from the corresponding sundifferential that are used  to form zero on their convex hull. Generally, there are several ways to form zero, but if we choose the one with the minimal number of maximal deviation points, then, indeed, the removal of any of the extremal points will lead to a situation where zero can not be formed anymore and the corresponding subsets of positive and negative points are isolable (their convex hulls do not intersect).

Therefore, our necessary and sufficient optimality conditions are equivalent to Rice's conditions. The main advantages of our formulations are as follows. First of all, our condition is much simpler, easier to understand and connect with the classical theory of univariate chebyshev approximation. Second, it is much easier to verify our optimality conditions, which is especially important for the construction of of a Remez-like algorithm, where necessary and sufficient optimality conditions need to be verified at each iteration.

 \section{Conclusions and further research directions}\label{sec:conclusions}
In this paper we obtained necessary and sufficient optimality conditions for best polynomial Chebyshev approximation (characterisation theorem). The main obstacle was to modify the notion of alternance to the case of multivatriate polynomials. This has been done using nonsmooth calculus and quasidifferentiability. We also propose an algorithm for optimality verification.

For the future we are planning to proceed in the following directions.
\begin{enumerate}
\item Find a necessary and sufficient optimality condition that is easy to verify in practice (current, we only have a necessary condition, but not a sufficient one).
\item Extend these results to the case of variable polynomial degrees for each dimension.
\item Develop similar optimality conditions for multivariate trigonometric polynomials and polynomial spline Chebyshev approximations.
\item Develop an approximation algorithm to construct best multivariate approximations (similar to the famous Remez algorithm, developed for univariate polynomials~\cite{remez57} and extended to  polynomial splines~\cite{nurnberger,sukhorukovaalgorithmfixed})
\end{enumerate}

    \bibliographystyle{amsplain}
    \bibliography{julien}

\end{document}